\documentclass[10pt]{article}
\usepackage{a4wide,amsmath,amsfonts,amsthm,amssymb,epsfig,psfrag,srcltx,url,amscd}
\date{}
\title{Upper bound for the first non-zero eigenvalue of the $p$-Laplacian}
\author{Sheela Verma\\Department of Mathematics \& Statistics\\ Indian Institute of
Technology Kanpur \\ Kanpur, India\\ sheela@iitk.ac.in}

\newtheorem{thm}{Theorem}
\newtheorem{lemma}{Lemma}

\newtheorem{rmk}[thm]{Remark}

\begin{document}
\maketitle
\begin{abstract}

 Let $M$ be a closed hypersurface in $\mathbb{R}^{n}$ and $\Omega$ be a bounded domain such that $M= \partial\Omega$. In this article, we obtain an upper bound for the first non-zero eigenvalue of the following problems.
 \begin{itemize}
\item Closed eigenvalue problem:
\begin{align*}
\Delta_p u = \lambda_{p} \ |u|^{p-2} \ u   \qquad \mbox{ on } \quad {M}.
\end{align*}
\item Steklov eigenvalue problem:
 \begin{align*}
\begin{array}{rcll}
 \Delta_{p}u &=& 0 & \mbox{ in } \Omega ,\\
|\nabla u|^{p-2} \frac{\partial u}{\partial \nu} &=& \mu_{p} \ |u|^{p-2} \ u &\mbox{ on } M .
\end{array}
\end{align*}
\end{itemize}

\end{abstract}

\section{Introduction}
The $p$-Laplace operator, defined as $ \Delta_{p} u := - \text{div}\left( {|\nabla u|^{p-2}} \nabla u \right)$, is the nonlinear generalization of the usual Laplace operator. 

Many interesting results, providing the sharp upper bounds for the first non zero eigenvalue of the usual Laplacian $\left( p=2\right) $ have been obtained. In \cite{BW}, Bleecker and Weiner obtained a sharp upper bound of the first non-zero eigenvalue of Laplacian in terms of the second fundamental form on a hypersurface $M$ in $\mathbb{R}^n$. In \cite{R}, Reilly gave an upper bound for the first non-zero eigenvalue in terms of higher order mean curvatures for a compact $n$-dimensional manifold isometrically immersed in $\mathbb{R}^{n+p}$, which improves the earlier estimate. This result was later extended to submanifolds of simply connected space forms in various ways  ( see \cite{EH, JFG}). These upper bounds are extrinsic in the sense that they depend either on the length of the second fundamental form or the higher order mean curvatures of $M$. 

Let $M$ be a hypersurface in a rank-$1$ symmetric space. In \cite{GS2}, an  upper bound for the first non-zero eigenvalue of $M$ was obtained in terms of the integral of the first non-zero eigenvalue of the geodesic spheres centered at the centre of gravity of $M$.

For a closed hypersurface $M$ contained in a ball of radius less than
$\frac{i(\mathbb{M}(k))}{4}$ and bounding a convex domain $Ω$ such that $\partial\Omega=M$ in the simply connected space form
$\mathbb{M}(k)$, $k=0$ or $1$, Santhanam \cite{GS1} proved that
\begin{align*}
\frac{\lambda_{1}(M)}{\lambda_{1}(S(R))} \leq \frac{\text{Vol}(M)}{\text{Vol}(S(R))},
\end{align*}
where $S(R)\left( = \partial B(R)\right) $ is the geodesic sphere of radius $R>0$ such that Vol$(B(R))$ = Vol$(\Omega)$. A similar result was also obtained for $k= -1$.

In this article, we extend the results in \cite{GS1} to $p$-Laplacian for a closed hypersurface $M \subset \mathbb{R}^n$. In particular, we consider the closed eigenvalue problem 
\begin{align}
\label{closedeigenvalueproblem}
\Delta_{p}u = \lambda_{p}\, |u|^{p-2}\,u \quad \mbox{ on } M,
\end{align} 
 where $M$ is a closed hypersurface in $\mathbb{R}^n$ and find an upper bound for the first non-zero eigenvalue of this problem.


Let ${M}$  be a closed hypersurface in $\mathbb{R}^n$ and $\Omega$ be the bounded domain such that ${M}  = \partial \Omega$. Consider the following problem
\begin{align*} 
\begin{array}{rcll}
\Delta f &=& 0 & \mbox{ in } \Omega ,\\
\frac{\partial f}{\partial \nu} &=& \mu f  &\mbox{ on } \partial \Omega,
\end{array}
\end{align*}
where $\nu$ is the outward unit normal on the boundary $\partial \Omega$ and $\mu$ is a real number.
This problem is known as Steklov eigenvalue problem and was introduced by Steklov \cite{S} for bounded domains in the plane in $1902.$ This problem is important as the set of eigenvalues of the Steklov problem is same as the set of eigenvalues of the well known Dirichlet-Neumann map.

There are several results which estimate the first non-zero eigenvalue $\mu_1$ of the Steklov eigenvalue problem \cite{P, E1, E2, E3, BGS}. 
The first isoperimetric upper bound for $\mu_{1}$ was given by Weinstock \cite{W} in $1954$. He proved that among all simply connected planar domains with analytic boundary of fixed perimeter, the circle maximizes $\mu_{1}$. In \cite{P}, Payne obtained a two sided bound for the first non-zero Steklov eigenvalue on a convex plain domain in terms of minimum and maximum curvature. The lower bound in \cite{P} has been generalized by Escobar \cite{E1} to $2$-dimensional compact manifold with non-negative Gaussian curvature. Using the Weinstock inequality, Escobar \cite{E2} proved that for a fix volume, among all bounded simply connected domain in $2$-dimensional simply connected space forms, geodesic balls maximize the first non-zero Steklov eigenvalue. This result has been extended to non-compact rank-$1$ symmetric spaces in \cite{BGS}. We prove the similar result for the first non-zero eigenvalue of the eigenvalue problem
\begin{align} \label{stekloveigenvalueproblem}
\begin{array} {rcll}
 \Delta_{p}u &=& 0 & \mbox{ in } \Omega, \\
|\nabla u|^{p-2} \frac{\partial u}{\partial \nu} &=& \mu_{p} \ |u|^{p-2} \ u &\mbox{ on } M,
\end{array}
\end{align}
where $\Omega$ is a bounded domain in $\mathbb{R}^n$ such that ${M}  = \partial \Omega$ and $\nu$ is outward unit normal on $M$. 

In Section $2$, we state our main results. In section $3$, we state some basic facts about the first non-zero eigenvalues of problem \eqref{closedeigenvalueproblem} and \eqref{stekloveigenvalueproblem}, and prove some results which will be required in the later sections. Followed by this, in section $4, 5$ and $6$, we provide the proof of results stated in section $2$.

\section{Statement of the results}
We state a variation of centre of mass theorem. This is crucial for our proof of main results.
\begin{thm}
\label{thm:test function}
Let $\Omega$ be a bounded domain in $\mathbb{R}^n$ and $M = \partial \Omega$. Then for every real number $1<p< \infty$, there exists a point $t \in \overline{\Omega} $ depending on $p$ and normal coordinate system $\left(X_1,X_2,\ldots,X_n \right)$ centered at $t$ such that for $ 1 \leq i \leq n$,
\begin{align*}
\int_M{|X_i|^{p-2} X_i} = 0.
\end{align*}

\end{thm}

Now we state our main results.

The following theorem provides an upper bound for the first non-zero eigenvalue $\lambda_{1,p}$ of the closed eigenvalue problem \eqref{closedeigenvalueproblem}.
\begin{thm}
 \label{thm:closedfe}
Let $M$ be a closed hypersurface in $\mathbb{R}^n$ bounding a bounded domain $\Omega$. Let $R>0$ be such that $\text{Vol }(\Omega)=\text{Vol }(B(R))$, where $B(R)$ is a ball of radius $R$. Then the first non-zero eigenvalue $\lambda_{1,p}$ of the closed eigenvalue problem~\eqref{closedeigenvalueproblem} satisfies 
\begin{align}
\label{eqn:closedfe1}
\lambda_{1,p} \leq {n}^\frac{\vert p-2 \vert}{2} \, {\lambda_1(S(R))}^\frac{p}{2} \, \left(\frac{\text{ Vol }(M)}{\text{ Vol }(S(R))}\right).
\end{align}

Furthermore, for $p=2$, the  upper bound (\ref{eqn:closedfe1}) is sharp and the equality holds if and only if $M$ is a geodesic sphere of radius $R$ ( see \cite{GS1}).

  If equality holds in \eqref{eqn:closedfe1} then $M$ is a geodesic sphere and $p=2$.
\end{thm}


In case of Steklov eigenvalue problem, we have the following upper bound for the first non-zero eigenvalue $\mu_{1,p}$.
\begin{thm}
\label{thm:steklovfe}
Let $\Omega$ be a bounded domain in $\mathbb{R}^n$ with smooth bounday $M$ and $R>0$ be such that $\text{ Vol }(\Omega)=\text{ Vol }(B(R))$, where $B(R)$ is a ball of radius $R$. Then the first non-zero eigenvalue $\mu_{1,p}$ of problem \eqref{stekloveigenvalueproblem} satisfies the following inequality.
\begin{itemize}
\item For $ 1<p<2$,
\begin{align}
\label{eqn:steklovfe2}
\mu_{1,p} \leq \frac{1}{R^{p-1}}.
\end{align}
\item For $p \geq 2$,
\begin{align}
\label{eqn:steklovfe1}
\mu_{1,p} \leq \frac{{n}^{p-2}}{R^{p-1}}.
\end{align}
\end{itemize} 

Furthermore, for $p=2$, equality holds in (\ref{eqn:steklovfe2}) and (\ref{eqn:steklovfe1}) iff $M$ is a geodesic sphere of radius $R$ (see \cite{BGS}).

If equality holds in \eqref{eqn:steklovfe2} and \eqref{eqn:steklovfe1} then $M$ is a geodesic sphere of radius $R$ and $p=2$.
\end{thm}




\section{Preliminaries}
In this section, we state some basic facts about the first non-zero eigenvalue of the eigenvalue problems \eqref{closedeigenvalueproblem} and \eqref{stekloveigenvalueproblem}. We will also prove some results that are needed in subsequent sections. 

Let $u_1$ be an eigenfunction corresponding to the eigenvalue $\lambda_p$ of closed eigenvalue problem \eqref{closedeigenvalueproblem} and $u_2$ be an eigenfunction corresponding to the eigenvalue $\mu_{p}$ of the Steklov eigenvalue problem \eqref{stekloveigenvalueproblem}. Then  $\lambda_p$ and $\mu_{p}$ satisfy
\begin{align*}
\lambda_{p} \int_M  {|{u_1}|^p} = \int_M   \|\nabla^{M}  u_1\|^p,
\end{align*}
\begin{align*}
\mu_{p} \int_M  {|{u_1}|^p} = \int_\Omega   \|\nabla  u_1\|^p.
\end{align*}
This shows that all eigenvalues of problems \eqref{closedeigenvalueproblem} and \eqref{stekloveigenvalueproblem} are non-negative.

Let $\lambda_{1,p}$ and $\mu_{1,p}$ be the first non-zero eigenvalues of the closed and steklov eigenvalue problems, respectively. Then the variational characterization for $\lambda_{1,p}$ and $\mu_{1,p}$ is given by
\begin{align*}
\lambda_{1,p}= \inf  \left\lbrace \frac{\int_{M}{\|\nabla^{M} u\|^p}}{\int_M{|u|^p}} : \int_M{|u|^{p-2} u} =0,  u(\neq 0) \in C^1(M)\right\rbrace,
\end{align*}
\begin{align*}
\mu_{1,p}= \inf \left\lbrace \frac{\int_\Omega{\|\nabla u\|^p}}{\int_M{|u|^p}} : \int_M{|u|^{p-2} u} =0,  u(\neq 0) \in C^1(\Omega)\right\rbrace.
\end{align*}

\begin{rmk} If $p=2$, then the condition $\int_M{|u|^{p-2} u} = \int_M{u} = 0 $ is equivalent to say that the test function must be orthogonal to the constant function in $L^2$-\text{norm}.
\end{rmk}

Let $M$ be a closed hypersurface in $\mathbb{R}^n$ and $\Omega$ be a bounded domain in $\mathbb{R}^n$ such that $M = \partial\Omega$. Fix a point $q \in \Omega$. Then for every point $s \in M$, the line joining $q$ and $s$ may intersect $M$ at some points other than $s$. For every point $s \in M$, let $r(s):=d(q,s)$ and for every $u \in {\mathbb{S}^{n-1}}$, let
$\beta(u):= \text{ max } \left\lbrace \beta >0 | \, {q+\beta u} \in M , \,  \beta \in \mathbb{R}\right\rbrace$. Let $A := \left\lbrace {q+ \beta(u)u}| \, u \in {\mathbb{S}^{n-1}} \right\rbrace$. Then $A \subseteq M$.
\begin{lemma}
\label{lem:lemma1}
Let $\Omega$ be a bounded domain in $\mathbb{R}^n$ with smooth boundary $M$ and $R>0$ be such that $\text{ Vol }(\Omega)=\text{Vol }(B(R))$, where $B(R)$ is a ball of radius $R$. Fix a point $q\in \overline{\Omega}$, then
\begin{align}
\label{eqn:lemma1}
\int_M {r^{p}(s)} \, ds \geq {R^{p}} \text{ Vol }(S(R)).
\end{align}
Further, equality holds in \eqref{eqn:lemma1} iff $M$ is a geodesic sphere of radius $R$ centered at $q$.
\end{lemma}
\begin{proof}  For a point $s \in A $, let $\gamma_s$ be the unique unit speed geodesic joining $q$ and $s$ with $\gamma_s(0)=q$. Let $u= \gamma'_s(0)$ and $t_s(u) = d(q,s) $.
Let $\theta(s)$ be the angle between the outward unit normal $\nu(s)$ to $M$ and the radial vector $\partial r(s)$. Let $du$ be the spherical volume density of the unit sphere ${\mathbb{S}^{n-1}}$. Then
\begin{align*}
\int_M {r^{p}(s)} \ ds &\geq \int_A {r^{p}(s)} \ ds \\
&= \int_{\mathbb{S}^{n-1}} \left(t_s(u)\right) ^{p} \sec \theta(s) \left(t_s(u)\right) ^{n-1} du \\
& \geq \int_{\mathbb{S}^{n-1}} \left(t_s(u)\right) ^{n+p-1} du \\
&= (n+p-1) \int_{\mathbb{S}^{n-1}} \int_{0} ^ {t_s(u)} {r}^{n+p-2} dr \ du     \\
& \geq (n+p-1) \int_\Omega r^{p-1} dV 
\end{align*}
and
\begin{align} \nonumber
\int_\Omega r^{p-1} dV &= \int_{\Omega\cap B(R)}r^{p-1} dV + \int_{\Omega\setminus{\Omega\cap B(R)}}r^{p-1} dV \\ \nonumber
&= \int_{ B(R)}r^{p-1} dV -  \int_{B(R)\setminus{\Omega\cap B(R)}}r^{p-1} dV +   \int_{\Omega\setminus{\Omega\cap B(R)}}r^{p-1} dV \\ \label{inequalityonr}
&\geq \int_{ B(R)}r^{p-1} dV -  \int_{B(R)\setminus{\Omega\cap B(R)}}r^{p-1} dV +   \int_{\Omega\setminus{\Omega\cap B(R)}}R^{p-1} dV \\ \nonumber
&= \int_{ B(R)}r^{p-1} dV +  \int_{B(R)\setminus{\Omega\cap B(R)}}(R^{p-1}-r^{p-1})dV  \\ \nonumber
& \geq \int_{ B(R)}r^{p-1} dV \\ \nonumber
&= \int_{\mathbb{S}^{n-1}} \int_{0} ^{R} r^{n+p-2} dr \ du \\ \nonumber
&= \int_{\mathbb{S}^{n-1}}\frac{R^{n+p-1}}{n+p-1} du    \\ \nonumber
&= \frac{R^{p}}{n+p-1} \text{ Vol }(S(R)).
\end{align}
We have used the fact that $R \leq r \text{ in } \left( \Omega\setminus{\Omega\cap B(R)}\right)$ in \eqref{inequalityonr}.

Further, equality holds in \eqref{eqn:lemma1}  iff $\sec \theta(s)=1 \text{ for all points } s  \in M $ and $\text{ Vol} \left( {B(R)\setminus{\Omega\cap B(R)}}\right)  =0 $. Note that $\sec \theta(s)=1$ iff $ \theta(s)=0 \text{ for all points } s  \in M $. Therefore outward unit normal $\nu(s)= \partial r(s)\, \text{ for all points } s  \in M $. This shows that $ \Omega= B(q,R)$ and $M$ is a geodesic sphere of radius $R$.\end{proof}

Above lemma is the generalization of the Lemma~1 in \cite{GS1}.

\begin{lemma}
\label{lem:lemma2}
Let $n \in \mathbb{N}$ and $y_1, y_2,\ldots, y_n $ be non-negative real numbers. Then for every real number $\gamma \geq 1$, the following inequality holds.
\begin{align}
\label{eqn:lemma2}
\left( y_1 + y_2 + \cdots + y_n \right)^ {\gamma} \geq y_{1}^{\gamma} + y_{2}^{\gamma} + \cdots + y_{n}^{\gamma} .
\end{align}
\end{lemma}
\begin{proof}  
Let $n \in \mathbb{N}$ and $y_1, y_2,\ldots, y_n $ be non-negative real numbers. Let $\gamma \geq 1$. Then inequality~\eqref{eqn:lemma2} can be written as
$$
\left( \frac{y_1}{ y_1 + y_2 + \cdots + y_n }\right) ^{\gamma} + \left( \frac{y_2}{ y_1 + y_2 + \cdots + y_n }\right) ^{\gamma} + \cdots + \left( \frac{y_n}{ y_1 + y_2 + \cdots + y_n }\right) ^{\gamma} \leq 1 .$$ 
Therefore, it is enough to show that
$a_{1}^{\gamma} + a_{2}^{\gamma}+\cdots + a_{n}^{\gamma} \leq 1$ for non-negative real numbers $ a_i$  such that $a_1+a_2+\cdots +a_n =1$.
 
Since $0\leq a_i \leq 1$ and $\gamma \geq 1$, then $a_{i}^{\gamma}\leq a_i$. Therefore,
$ a_{1}^{\gamma} + a_{2}^{\gamma}+\cdots + a_{n}^{\gamma} \leq a_1 +a_2 +\cdots+a_n  =1$. This proves the Lemma. 
\end{proof}

Next we estimate $ {\sum_{i=1}^{n}} \|{\nabla ^M {x_i}}\|^{2} $. 

\begin{lemma}
\label{lem:lemma3}
Let ${M}$  be a closed hypersurface in $\mathbb{R}^n$ and $\Omega$ be a bounded domain such that ${M}  = \partial \Omega$. For a fixed point $t\in \overline{\Omega}$, let $(x_1, x_2,\ldots, x_n)$ be the normal coordinate system centered at $t$. Then
$$
{\sum_{i=1}^{n}} \|{\nabla ^M {x_i}}\|^{2} = (n-1).
$$
\end{lemma}
\begin{proof} Observe that $\|{\nabla {x_i}(p)}\| = 1$ for $1 \leq i \leq n$ and a point $p \in \mathbb{R}^n$. Let $\nu$ be the outward unit normal on $M$. Then
\begin{align*}
{\sum_{i=1}^{n}} \|{\nabla ^M {x_i}}\|^{2} &= {\sum_{i=1}^{n}} \left(  \|{\nabla {x_i}}\|^{2} - {\langle \nabla {x_i} , \nu \rangle}^2 \right) \\
&= {\sum_{i=1}^{n}}  \|{\nabla {x_i}}\|^{2} - \|\nu\|^2 \\
&= n-1.
\qedhere
\end{align*} 
\end{proof}

For a Riemannian geometric proof of above lemma, see \cite{EH}.

\section{Proof of Theorem \ref{thm:test function}}

\begin{proof} Given a point $x \in \mathbb{R}^n$, we write  $\left(x_1,\ldots,x_n \right)$, the standard Euclidean coordinate system centered at origin. For $1<p<\infty$, define a function   $f: \overline{\Omega} \to \mathbb{R}$  by
$$
f\left( t_1, \ldots,t_n\right) =\frac{1}{p} \int_M{\sum_{i=1}^{n} |x_i - t_i|^{p}} \, dx_1 \, \cdots \, dx_n.
$$
The function $f$ is non-negative on $\overline{\Omega}$. Let $\alpha$ be its infimum. Then there exists a sequence $\left(t_{1}^{j}, \ldots,t_{n}^{j} \right) $ in $\overline{\Omega}$ such that
\begin{align}
\label{eqn:test function1}
\frac{1}{p} \int_M{\sum_{i=1}^{n} |x_i - t_{i}^{j}|^{p}} \longrightarrow \alpha \quad \text{ in } \mathbb{R} \quad \text{ as } j \longrightarrow \infty.
\end{align}
Observe that the sequence $\left( t_{1}^{j}, \ldots, t_{n}^{j}\right) $ is bounded. Therefore it has a convergent subsequence, without loss of generality, we denote it by $\left( t_{1}^{j}, \ldots, t_{n}^{j}\right) $ itself, which converges to  $t = \left( t_1,\ldots,t_n\right) \in \mathbb{R}^n$. Thus $t \in \overline{\Omega}$. Then

\begin{align*}
 \sum_{i=1}^{n} |x_i - t_{i}^{j}|^{p} &\longrightarrow   \sum_{i=1}^{n} |x_i - t_{i}|^{p} \qquad as \quad j \longrightarrow \infty \\
\mbox{ and } \qquad \frac{1}{p} \int_M {\sum_{i=1}^{n} |x_i - t_{i}^{j}|^{p}} &\longrightarrow \frac{1}{p} \int_M{\sum_{i=1}^{n} |x_i - t_{i}|^{p}} \qquad as \quad j \longrightarrow \infty.
\end{align*}
Therefore,

$$\frac{1}{p} \int_M {\sum_{i=1}^{n} |x_i - t_{i}|^{p}} = \alpha \qquad \mbox{ and }  \qquad
f\left(t_1,\ldots,t_n \right) = \alpha.$$ 

Since $f$ attains its minimum at $t= \left(t_1,\ldots,t_n \right)$, we have  $\left(\nabla f \right)_{t} = 0$. Therefore for each $1\leq i\leq n$,
\begin{align*}
\langle \nabla f, e_i \rangle_{(t_1,\ldots,t_n)} = \int_{M} |X_i|^{p-2} X_i =0,
\end{align*}
where $\left\lbrace e_i, 1 \leq i \leq n\right\rbrace $ is the standard orthonormal basis of $\mathbb{R}^n$ and $X_i:= (x_i - t_i), 1 \leq i \leq n$.
This proves the theorem. 
\end{proof}
 We will use the above theorem to show the existence of a point ${t}=\left( t_1,\ldots,t_n\right)\in \overline{\Omega} $, such that the coordinate functions with respect to $t$ are test functions for the eigenvalue problems \eqref{closedeigenvalueproblem} and \eqref{stekloveigenvalueproblem}.

\section{Proof of Theorem \ref{thm:closedfe}}
\begin{proof} Let ${M}$  be a closed hypersurface in $\mathbb{R}^n$ and $\Omega$ be the bounded domain such that ${M}=\partial \Omega$. Let $R>0$ be such that $\text{Vol }(\Omega)=\text{ Vol }(B(R))$.
The variational characterization for $\lambda_{1,p}$ is given by
\begin{align*}
\lambda_{1,p}= \text{ inf } \left\lbrace \frac{\int_{M}{\|{\nabla ^M {u}}\|^p}}{\int_M{|u|^p}} : \int_M{|u|^{p-2} u =0, u (\neq 0)\in C^1(M)}\right\rbrace.
\end{align*}
By Theorem \ref{thm:test function}, there exists a point $t \in \overline{\Omega}$  such that
\begin{align*}
\int_M{|x_i|^{p-2} x_i} = 0 \qquad \text{ for } 1 \leq i\leq n,
\end{align*}
where $\left(x_1,\ldots,x_n \right)$  denotes the normal coordinate system centered at $t$. 
Therefore, for all $p>1$,
\begin{align}
\label{eqn:closedfep1}
\lambda_{1,p} \int_M  {\sum_{i=1}^{n}|{x_i}|^p} \leq \int_M  \sum_{i=1}^{n} \|{\nabla ^M {x_i}}\|^p \quad \text{ for } 1 \leq i\leq n.
\end{align}
Now, we divide the proof of the theorem into the following two cases.\\
$\mathbf{Case \, 1}$.  $1 < p \leq 2$. \\
 Since $\vert \frac{x_i}{r}\vert \leq 1$, it follows that
\begin{align}
\label{eqn:closedfep4}
|x_i|^p = {r^p \left|\frac{x_i}{r}\right|^p} \geq {r^p \left|{\frac{x_i}{r}}\right|^2 } \quad \text{ for } 1 \leq i\leq n.
\end{align}
Therefore,
\begin{align*}
r^{p} = r^{p}\sum_{i=1}^{n} \left|\frac{x_i}{r}\right|^2 \leq r^{p}\sum_{i=1}^{n} \left|\frac{x_i}{r}\right|^{p} = \sum_{i=1}^{n} {|x_i|^p}.
\end{align*}
For $1 < p < 2$, using H$\ddot{o}$lder's inequality, we obtain
\begin{align*}
\sum_{i=1}^{n} {\|{\nabla ^M {x_i}}\|^p} \leq  \left(  {\sum_{i=1}^{n}  \|{\nabla ^M {x_i}}\|^2 }\right)  ^ \frac{p}{2}  n^\frac{2-p}{2}.
\end{align*}
This combining with  Lemma \ref{lem:lemma3} gives
$$
\sum_{i=1}^{n} {\|{\nabla ^M {x_i}}\|^p}\leq (n-1)^{\frac{p}{2}} n^\frac{2-p}{2}.
$$
Observe that the above inequality is also true for $p=2$. 
By substituting above values in inequality \eqref{eqn:closedfep1}, we get
\begin{align}
\label{eqn:closedfep5}
\lambda_{1,p} {\int_M {r^p}} \leq (n-1)^{\frac{p}{2}} \ {n^\frac{2-p}{2}} \, \text{Vol}(M).
\end{align}
By substituting  ${\int_M {r^{p}}} \geq {R^P} \,\text{Vol}(S(R)) $ in above equation, we get
$$
\lambda_{1,p} \, {R^P} \, \text{Vol}(S(R))  \leq  (n-1)^{\frac{p}{2}} \, {n^\frac{2-p}{2}} \, \text{Vol}(M). $$
As a consequence, we have
$$
\lambda_{1,p}  \leq  {n}^\frac{2-p} {2}\, {\lambda_1(S(R))}^\frac{p}{2} \, \left(\frac{\text{Vol}(M)}{\text{Vol}(S(R))}\right).
$$
This proves Theorem \ref{thm:closedfe} for $1 < p \leq 2$. 

Equality in (\ref{eqn:closedfe1}) implies equality in Lemma \ref{lem:lemma1} and equality in (\ref{eqn:closedfep4}), which implies that $M$ is a geodesic sphere of radius $R$ and $p=2$. \\
$\mathbf{Case \, 2}$. $p \geq 2$. \\
 For $p > 2$,
by H$\ddot{o}$lder's inequality, we have
\begin{align*}
 \sum_{i=1}^{n} {|x_i|^2} &\leq  \left(  {\sum_{i=1}^{n} \left( |x_i|^2 \right) ^ \frac{p}{2}}\right)  ^ \frac{2}{p}  n^\frac{p-2}{p}.
\end{align*}
Therefore,
\begin{align}
\label{eqn:closedfep2}
  n^\frac{2-p}{2} {r}^p &\leq  \sum_{i=1}^{n} {|x_i|^p}.
\end{align}
Observe that equality holds in the above inequality for $p=2$, so  (\ref{eqn:closedfep2}) holds for $p \geq 2$. 
Now we estimate $\sum_{i=1}^{n} \|{\nabla ^M {x_i}}\|^p$. Since $\frac{p}{2} \geq 1$ and $\|{\nabla ^M {x_i}}\|^2 \geq 0$, for each $1 \leq i \leq n$, it follows from Lemma \ref{lem:lemma2} that
\begin{align} \nonumber
 \sum_{i=1}^{n} \|{\nabla ^M {x_i}}\|^p &= {\sum_{i=1}^{n} \left( \|{\nabla ^M {x_i}}\|^2\right) ^\frac{p}{2}} \\ \nonumber
 &\leq    \left(  \sum_{i=1}^{n} \|{\nabla ^M {x_i}}\|^2\right)                                                                                                                                                                                                                                                                                                                                                                                                                                                                    ^\frac{p}{2} \\  \label{eqn:closedfep6}
 &= {(n-1)}^{\frac{p}{2}}.
\end{align}
The last inequality follows from Lemma \ref{lem:lemma3}. By substituting values from (\ref{eqn:closedfep2}) and (\ref{eqn:closedfep6}) in (\ref{eqn:closedfep1}), we get

\begin{equation}
\label{eqn:closedfep3}
\lambda_{1,p}\,{n^\frac{2-p}{2}}{\int_M {r^p}} \leq (n-1)^\frac{p}{2} \, \text{Vol}(M).
\end{equation}

 By substituting ${\int_M {r^{p}}} \geq {R^p} \,\text{Vol}(S(R)) $ from Lemma \ref{lem:lemma1} in above inequality, we have
 $$
 \lambda_{1,p}\,{n^\frac{2-p}{2}}\,{R^p} \,\text{Vol}(S(R)) \leq (n-1)^\frac{p}{2} \, \text{Vol}(M).$$ 
 Therefore,
 $$
\lambda_{1,p}  \leq  {n}^\frac{p-2} {2}\, {\lambda_1(S(R))}^\frac{p}{2} \, \left(\frac{\text{Vol}(M)}{\text{Vol}(S(R))}\right).
$$ 
This proves Theorem \ref{thm:closedfe} for $p \geq 2 $.

If equality holds in (\ref{eqn:closedfe1}), then equality holds in \eqref{eqn:lemma1} and also in  \eqref{eqn:closedfep2}. Equality in \eqref{eqn:lemma1} implies that $M$ is a geodesic sphere of radius $R$ and equality in \eqref{eqn:closedfep2} holds iff $p=2$. Otherwise, $p > 2$ and equality in \eqref{eqn:closedfep2} implies that
$|x_i| = c$, for some constant $c$ and $1\leq i\leq n$.
Therefore, each point of $M$ is of the form $(\pm c,\pm c,\pm c,\ldots,\pm c)$, for some constant $c$. This contradicts our assumption that $M$ is the boundary of a bounded domain $\Omega$.
\qedhere
\end{proof}




\section{Proof of Theorem \ref{thm:steklovfe}}
\begin{proof} Let $\Omega$ be a bounded  domain in $\mathbb{R}^n$ with smooth bounday $\partial \Omega = M$ and $R>0$ be such that $\text{Vol}(\Omega)=\text{Vol}(B(R))$, where $B(R)$ is a ball of radius $R$. The variational characterization for $\mu_{1,p}$ is given by
\begin{align*}
\mu_{1,p}= \text{ inf } \left\lbrace \frac{\int_\Omega{\|\nabla u\|^p}}{\int_M{|u|^p}} : \int_M{|u|^{p-2} u =0, u(\neq 0) \in C^1(\Omega)}\right\rbrace.
\end{align*} 
By Theorem \ref{thm:test function}, there exists a point $t \in \overline{\Omega}$  such that
\begin{align*}
\int_M{|x_i|^{p-2} x_i} = 0, \qquad \text{for all} \quad 1 \leq i \leq n,
\end{align*}
where $\left(x_1,x_2,x_3,\ldots,x_n \right)$  denotes the normal coordinate system centered at $t$. 
By considering each ${x_i}$ as test function, we have
\begin{align}
\label{eqn:steklovfep1}
\mu_{1,p} \int_M  {\sum_{i=1}^{n}|{x_i}|^p} \leq \int_\Omega  \sum_{i=1}^{n} \|\nabla  x_i\|^p.
\end{align}
Now we consider the following two cases to prove the theorem. \\
$\mathbf{Case \, 1}$.  $1 < p \leq 2$. \\
By similar argument as in (\ref{eqn:closedfep4}), we get
\begin{align*}
r^p &\leq \sum_{i=1}^{n} {|x_i|^p}.
\end{align*}
By H$\ddot{o}$lder's inequality,
\begin{align*}
\sum_{i=1}^{n} {\|\nabla {x_i}\|^p} &\leq  \left(  {\sum_{i=1}^{n}  \|\nabla{x_i}\|^2 }\right)  ^ \frac{p}{2}  n^\frac{2-p}{2} = n.
\end{align*}
By substituting above values in \eqref{eqn:steklovfep1}, we get
\begin{align*}
\mu_{1,p}  \int_{M} {r}^p \leq n \ \text{ Vol}(\Omega).
\end{align*}
By substituting ${\int_M {r^{p}}} \geq {R^P} \,\text{Vol}(S(R)) $ from Lemma \ref{lem:lemma1}, we have
\begin{align*}
\mu_{1,p} \ {R}^p \, \text{Vol}(S(R)) \leq n \, \text{Vol}(\Omega).
\end{align*}
Since $ \text{Vol}(\Omega)=  \text{Vol}(B(R))$ and $\frac{\text{ Vol}(B(R))}{\text{ Vol}(S(R))} = \frac{R}{n}$, we get
\begin{align*}
\mu_{1,p} \leq \frac{1}{R^{p-1}}.
\end{align*}
$\mathbf{Case \, 2}$. $p \geq 2$. \\
From (\ref{eqn:closedfep2}), we have
\begin{align*}
 n^\frac{2-p}{2} {r}^p &\leq  \sum_{i=1}^{n} {|x_i|^p} \quad \text{  for all } \, p \geq 2.
\end{align*}
By Lemma \ref{lem:lemma2}, we have
\begin{align*}
\sum_{i=1}^{n} \|\nabla x_i\|^p &\leq {\sum_{i=1}^{n} \left( \|\nabla x_i\|^2\right) ^\frac{p}{2}} \\
 &\leq    \left(  \sum_{i=1}^{n} \|\nabla x_i\|^2\right)^\frac{p}{2} \\
 &\leq {n}^{\frac{p}{2}}.
\end{align*}
By substituting above values in (\ref{eqn:steklovfep1}), we get
\begin{align*}
\mu_{1,p} \ n^\frac{2-p}{2} \int_{M} {r}^p \leq n^\frac{p}{2} \ \text{ Vol }(\Omega).
\end{align*}
We use Lemma \ref{lem:lemma1} again to get
\begin{align*}
\mu_{1,p} \ n^\frac{2-p}{2} R^p \ \text{Vol}(S(R)) \leq n^\frac{p}{2} \ \text{Vol}(\Omega).
\end{align*}
Since $\text{Vol}(\Omega) = \text{Vol}(B(R))$ and $\frac{\text{Vol}(B(R))}{\text{Vol}(S(R))} = \frac{R}{n}$, above equation becomes
\begin{align*}
\mu_{1,p} \leq \frac{n^{p-2}}{R^{p-1}}.
\end{align*}
Equality case will follow same as in Theorem \ref{thm:closedfe}.
This completes the proof.
\end{proof}

\section*{Acknowledgment}
 I am very grateful to Prof. G. Santhanam for his support and constructive suggestions which led to improvements in the article.


\begin{thebibliography}{9}

\bibitem{S} M. W. Stekloff, les problemes fondamentaux de la physique mathematique, \emph{Ann. Sci. Ecole Norm} \textbf{19}  445-490 (1902). 
\bibitem{W} R. Weinstock, Inequalities for a classical Eigenvalue problem, \emph{ Rational Mech. Anal} \textbf{3}  745-753 (1954).
\bibitem{P} L. E. Payne, Some isoperimetric inequalities for Harmonic functions, \emph{SIAM J. Math. Anal.} \textbf{1} 354-359 (1970).

\bibitem{BW} D. Bleecker, J. Weiner, Extrinsic bounds on $\lambda_1$ of $\Delta$ on a compact manifold, \emph{Comment. Math. Helv.} \textbf{51} 601-609 (1976).
\bibitem{R} R. Reilly, On the first eigenvalue of the Laplacian for compact submanifold of Euclidean space, \emph{Comment. Math. Helv.} \textbf{52} 525-533 (1977).
\bibitem{EH} E. Heintze, Extrinsic upperbounds for $\lambda_1$, \emph{Math. Ann.} \textbf{280} 389-402 (1988).
\bibitem{E1}J. F. Escobar, The geometry of the first non-zero Stekloff eigenvalue, \emph{Journal of Functional Analysis} \textbf{150}(2) 544-556 (1997). 
\bibitem{E2}J. F. Escobar, An Isoperimetric Inequality and the First Steklov Eigenvalue, \emph{ Journal of Functional Analysis} \textbf{165}(1)  101-116 (1999).
\bibitem{E3}J. F. Escobar, A Comparison Theorem for the First Non-zero Steklov Eigenvalue, \emph{ Journal of Functional Analysis} \textbf{178}(1)  143-155 (2000).

\bibitem{JFG} J. F. Grosjean, Upper bounds for the first eigenvalue of the Laplacian on compact submanifolds, \emph{Pacific. J. Math.} \textbf{206} 93-112 (2002).
\bibitem{GS2} Santhanam G, A sharp upper bound for the first eigenvalue of the Laplacian of compact hypersurfaces in rank-$1$ symmetric spaces, \emph{Proc. Indian Acad. Sci. (Math. Sci.)} \textbf{117}(3) 307-315 (2007).
\bibitem{GS1} G. Santhanam, Isoperimetric upper bounds for the first eigenvalus, \emph{Proc. Indian Acad. Sci. (Math. Sci.)} \textbf{122}(3) 375-384 (2012).
\bibitem{BGS} Binoy, G. Santhanam,  Sharp upperbound and a comparison theorem for the first nonzero Steklov eigenvalue, \emph{J. Ramanujan Math. Soc.} \textbf{29}(2) 133-154 (2014).
\end{thebibliography}
\end{document}